\documentclass[12pt]{amsart}
\usepackage{amsmath,amssymb,enumerate}
\newcommand{\IB}{{\mathbb B}}
\newcommand{\IN}{{\mathbb N}}

\newcommand{\cH}{{\mathcal H}}

\newcommand{\cU}{{\mathcal U}}
\newcommand{\acts}{\curvearrowright}
\DeclareMathOperator{\supp}{supp}
\DeclareMathOperator{\Ad}{Ad}
\newcommand{\ip}[1]{\mathopen{\langle}#1\mathclose{\rangle}}
\newtheorem*{thm*}{Theorem}
\newtheorem*{prop*}{Proposition}
\newtheorem*{lem*}{Lemma}
\newtheorem{thmA}{Theorem}

\title[Gromov's polynomial growth theorem]{A functional analysis proof of Gromov's polynomial growth theorem}
\author{Narutaka Ozawa}
\address{RIMS, Kyoto University, \mbox{606-8502} Japan}
\email{narutaka@kurims.kyoto-u.ac.jp}
\thanks{The author was partially supported by JSPS26400114}
\subjclass{20F65; 60B15, 43A07}

\keywords{Reduced cohomology, hamonic $1$-cocycles}
\date{2016 March 12}

\begin{document}
\begin{abstract}
The celebrated theorem of Gromov asserts that any finitely generated group 
with polynomial growth contains a nilpotent subgroup of finite index. 
Alternative proofs have been given by Kleiner and others. 
In this note, we give yet another proof of Gromov's theorem, 
along the lines of Shalom and Chifan--Sinclair, which is based on 
the analysis of reduced cohomology and Shalom's property $H_{\mathrm{FD}}$. 
\end{abstract}
\maketitle
\section{Introduction}
The celebrated theorem of Gromov (\cite{gromov,vw}) asserts that any finitely generated group 
with weakly polynomial growth contains a nilpotent subgroup of finite index. 
Here a group $G$ is said to have \emph{weakly polynomial growth} if 
$\liminf \log|S^n|/\log n<\infty$ for any finite generating subset $S$ such that $1\in S=S^{-1}$. 
Alternative proofs have been 
given by Kleiner and others (\cite{kleiner,shalom-tao,hrushovski,bgt}). 
In this note, we give yet another proof of Gromov's theorem, along the lines of 
Shalom (\cite{shalom-acta}) and Chifan--Sinclair (\cite{chifan-sinclair}),
which is based on the analysis of reduced cohomology and Shalom's property $H_{\mathrm{FD}}$. 

Let $\pi\colon G\acts\cH$ be a unitary representation. 
Recall that a \emph{$1$-cocycle} of $G$ with coefficients in $\pi$ 
is a map $b\colon G \to \cH$ which satisfies 
\[
\forall g,x\in G\quad b(gx)=b(g)+\pi_gb(x).
\] 
A \emph{$1$-coboundary} is a $1$-cocycle of the form $b(g)=\xi-\pi_g\xi$ for some $\xi\in\cH$, 
and an \emph{approximate $1$-coboundary} is a $1$-cocycle that is a pointwise limit of 
$1$-coboundaries. 
The spaces of $1$-cocycles, $1$-coboundaries, approximate $1$-coboundaries are 
written respectively by $Z^1(G,\pi)$, $B^1(G,\pi)$, and $\overline{B^1(G,\pi)}$, 
and so the \emph{reduced cohomology} space $\overline{H^1}(G,\pi)$ is 
$Z^1(G,\pi)/\overline{B^1(G,\pi)}$.
It is proved by Mok and Korevaar--Schoen 
(\cite{mok,ks}, see also \cite{shalom-invent} and 
Theorem~\ref{thmA} in Appendix) that 
any finitely generated group $G$ without Kazhdan's property (T) 
admits a unitary representation $\pi$ with $\overline{H^1}(G,\pi)\neq0$. 
A group $G$ is said to have \emph{Shalom's property $H_{\mathrm{FD}}$} if 
$\overline{H^1}(G,\pi)\neq0$ implies that $\pi$ is not weakly mixing.
Here $\pi$ is said to be \emph{weakly mixing} if $\cH$ admits no nonzero 
finite-dimensional $\pi(G)$-invariant subspaces. 
We recall that infinite amenable groups, and in particular groups with weakly polynomial growth, 
do not have property (T) (see, e.g., \cite[Chapter 12]{bo}). 
Thus, if such a group has property $H_{\mathrm{FD}}$, then it has a finite-dimensional 
unitary representation $\pi$ with $\overline{H^1}(G,\pi)\neq0$. 
Shalom has observed that a proof of property $H_{\mathrm{FD}}$ for 
a group with weakly polynomial growth implies Gromov's theorem 
(see \cite[Section 6.7]{shalom-acta} and \cite{tao-blog}). 
In this paper, we prove that a group with slow entropy growth 
has property $H_{\mathrm{FD}}$, thus giving a new proof of Gromov's theorem. 
Here we say $G$ has \emph{slow entropy growth} 
if there is a non-degenerate finitely-supported symmetric probability measure $\mu$ on $G$ 
with $\mu(e)>0$ such that  
\[
\liminf_n n(H(\mu^{*n+1})-H(\mu^{*n}))<\infty,
\] 
where $H$ is the entropy functional.
This property is formerly weaker than but probably equivalent 
to weakly polynomial growth (see Section~\ref{sec:entropy}).

\begin{thm*}\label{thm}
A finitely generated group with slow entropy growth has property $H_{\mathrm{FD}}$.  
\end{thm*}

\subsection*{Acknowledgment}
The author would like to thank Professor L. Saloff-Coste for drawing his attention to 
\cite{ek}, Professor A. Erschler for encouraging the author to include 
Proposition in Section~\ref{sec:proof}, and Professor A. Yadin for useful comments 
on the slow entropy growth condition. 
\section{Reduced cohomology and harmonic $1$-cocycles}\label{sec:harmonic}
Let $G$ be a finitely generated group and fix a non-degenerate finitely-supported 
symmetric probability measure $\mu$ with $\mu(e)>0$. 
Let $\pi\colon G\acts\cH$ be a unitary representation. 
We first recall the fact that every element in the reduced cohomology 
space $\overline{H^1}(G,\pi)$ is uniquely represented by 
a $\mu$-harmonic $1$-cocycle (see \cite{guichardet,bv}).
The space $Z^1(G,\pi)$ of $1$-cocycles is a Hilbert space under the norm 
\[
\| b \|_{Z^1(G,\pi)} := \Bigl(\sum_x \mu(x)\| b(x) \|^2\Bigr)^{1/2}, 
\]
and the space $\overline{B^1(G,\pi)}$ agrees with the closure of $B^1(G,\pi)$ 
in the Hilbert space $Z^1(G,\pi)$.
We observe that $b\in Z^1(G,\pi)$ is orthogonal to $B^1(G,\pi)$ if and only if 
it is \emph{$\mu$-harmonic}: $\sum_x \mu(x)b(x)=0$ 
or equivalently $\sum_x \mu(x)b(gx) = b(g)$ for all $g\in G$. 
Indeed, this follows from the identities $b(x^{-1})+\pi_x^{-1}b(x)=b(e)=0$ and 
\[
\sum_x \mu(x)\ip{b(x),\xi-\pi_x\xi} = 2\ip{\sum_x \mu(x)b(x), \xi}.
\]
Since $Z^1(G,\pi)=\overline{B^1(G,\pi)}\oplus B^1(G,\pi)^\perp$ as a Hilbert space, 
$\overline{H^1}(G,\pi)$ can be identified with the space 
$B^1(G,\pi)^\perp$ of $\mu$-harmonic $1$-cocycles.

By the above discussion, we may concentrate on $\mu$-harmonic $1$-cocycles. 
For any $\mu$-harmonic $1$-cocycle $b$, 
one has $\sum_x \mu^{*n}(x)\| b(x) \|^2 = n \sum_x \mu(x)\|b(x)\|^2$ 
(by induction on $n$). 
In this section, we give a better inequality, 
which is inspired by the work of Chifan and Sinclair (\cite{chifan-sinclair}). 
Let $\cH\otimes\bar{\cH}$ denote the Hilbert space tensor product of 
the Hilbert space $\cH$ and its complex conjugate $\bar{\cH}$. 
We recall that $\pi$ is weakly mixing if and only if the unitary representation 
$\pi\otimes\bar{\pi}$ on $\cH\otimes\bar{\cH}$ has no nonzero invariant vectors. 
Indeed, $\cH\otimes\bar{\cH}$ can be identified with the space $S_2(\cH)$ of 
Hilbert--Schmidt operators on $\cH$, and under this identification 
$\pi_g\otimes\bar{\pi}_g$ becomes the conjugation action $\Ad\pi_g$ of $\pi_g$ on $S_2(\cH)$
(see e.g.\ Section 13.5 in \cite{bo}).
Since any nonzero Hilbert--Schmidt operator (which is $\Ad\pi_g$-invariant) 
is compact and has a nonzero finite-dimensional 
eigenspace (which is $\pi_g$-invariant), our claim follows. 

\begin{lem*}
Let $b\colon G\to\cH$ be a $\mu$-harmonic $1$-cocycle with coefficients in a weakly mixing 
unitary representation $\pi$. Then one has
\[
\frac{1}{n}\Bigl\| \sum_x \mu^{*n}(x)(b(x)\otimes\bar{b}(x)) \Bigr\|_{\cH\otimes\bar{\cH}} \to 0.
\]
In particular, 
\[
\sup_{\xi\in\cH,\,\|\xi\|\le1}\frac{1}{n}\sum_x \mu^{*n}(x)|\ip{ b(x),\xi }|^2 \to 0.
\]
\end{lem*}
\begin{proof}
Since $b$ is $\mu^{*n}$-harmonic for every $n$, 
one has for every $n$ and $g\in G$ 
\[
\sum_x \mu^{*n}(x)(b(gx)\otimes\bar{b}(gx)) 
 = b(g)\otimes\bar{b}(g) + (\pi_g\otimes\bar{\pi}_g)\sum_x \mu^{*n}(x)(b(x)\otimes\bar{b}(x)).
\]
Thus, putting $\zeta := \sum_x \mu(x)(b(x)\otimes\bar{b}(x))$ and 
$T:=\sum_g \mu(g)(\pi_g\otimes\bar{\pi}_g)$, one has 
\begin{align*}
\sum_x \mu^{*n}(x)(b(x)\otimes\bar{b}(x))
 &= \sum_{g,x} \mu(g)\mu^{*n-1}(x) (b(gx)\otimes\bar{b}(gx))\\
 &= \zeta + T \sum_x \mu^{*n-1}(x) (b(x)\otimes\bar{b}(x)) \\ 
 &= \cdots =(1+T+\cdots+T^{n-1})\zeta.
\end{align*}
Since $\pi$ is weakly mixing, $\pi\otimes\bar{\pi}$ admits no nonzero invariant vectors, 
and hence by strict convexity of a Hilbert space, 
$1$ is not an eigenvalue of the self-adjoint contraction $T$. 
Hence, the measure $m(\,\cdot\,):=\ip{E_T(\,\cdot\,)\zeta,\zeta}$, associated 
with the spectral resolution $E_T$ of $T$, is supported on $[-1,1]$ and 
satisfies $m(\{1\})=0$. Thus, one has 
\[
\frac{1}{n}\Bigl\| \sum_x \mu^{*n}(x)(b(x)\otimes\bar{b}(x)) \Bigr\|_{\cH\otimes\bar{\cH}}
 = \Bigl(\int_{-1}^1\Bigl|\frac{1+t+\cdots+t^{n-1}}{n}\Bigr|^2\,dm(t)\Bigr)^{1/2}\to0
\]
by Bounded Convergence Theorem.
The second statement follows from the first, because 
$|\ip{b(x),\xi}|^2 = \ip{b(x)\otimes\bar{b}(x),\xi\otimes\bar{\xi}}$.
\end{proof}

\section{Concavity of the entropy functional}\label{sec:entropy}

We recall the basic fact that the entropy functional 
\[
p\mapsto H(p):=-\sum_x p(x)\log p(x) = \sum_x p(x)\log\frac{1}{p(x)}
\]
is concave and examine its modulus of concavity. 
Our discussion in this section is inspired by Erschler and Karlsson's work \cite{ek}. 
See also \cite{bdcky} for relevant information.
Let $p$ and $q$ be any non-negative functions. 
For convenience, we put 
\[
\delta(p,q) := H( \frac{p+q}{2} ) - \frac{H(p)+H(q)}{2}.
\]
Since 
\[
\frac{1}{2}( a\log a + b \log b ) - \frac{a+b}{2}\log\frac{a+b}{2} \geq \frac{| a - b |^2}{8(a+b)} \geq 0
\]
for any $a,b\geq0$ (this follows from the fact $(t\log t)''=(1+\log t)'= t^{-1} \geq (a+b)^{-1}$ for all 
$t$ between $a$ and $b$), one has 
\[
\delta(p,q) \geq \sum_x \frac{| p(x) - q(x) |^2}{8(p(x)+q(x))}\geq 0.
\]
This implies concavity of $H$. Moreover, for any non-negative function $f$, one has 
\begin{equation}\label{eq1}
\sum_x f(x)|p-q|(x) 
 \le \bigl( 8\delta(p,q)\sum_x f(x)^2(p+q)(x)  \bigr)^{1/2}
\end{equation}
by the Cauchy--Schwarz inequality.
In particular, $\| p - q \|_1 \le (8\delta(p,q)\| p+ q \|_1)^{1/2}$.

For any probability measures $\mu$ and $\nu$ on $G$ and $g_0\in G$, one has 
\begin{equation}\label{eq2}
H(\mu*\nu) - H(\nu) \geq 2\min\{ \mu(e),\mu(g_0)\} \delta(\nu,g_0\nu).
\end{equation}
Here $\mu*\nu=\sum_g\mu(g)(g\nu)$ and $(g\nu)(x)=\nu(g^{-1}x)$. 
Indeed, put $\lambda:=\min\{ \mu(e),\mu(g_0)\}$ and observe that 
$\nu':=(1-2\lambda)^{-1}(\mu*\nu - (\lambda\nu +\lambda g_0\nu))$ is a 
convex combination of $g\nu$'s, and that $H(g\nu)=H(\nu)$ for any $g$.
Hence, $H(\nu')\geq H(\nu)$ by concavity and 
\begin{align*}
H(\mu*\nu) - H(\nu) 
 &= H(2\lambda \frac{\nu+g_0\nu}{2} + (1-2\lambda)\nu') - H(\nu)\\
 &\geq 2\lambda \bigl(H(\frac{\nu+g_0\nu}{2}) - H(\nu)\bigr) + (1-2\lambda)\bigl( H(\nu')- H(\nu)\bigr)\\
 &\geq 2\lambda \delta(\nu,g_0\nu). 
\end{align*}

Here we explain that a group $G$ with weakly polynomial growth 
has slow entropy growth. 
Let $\mu$ be any non-degenerate finitely-supported symmetric 
probability measure on $G$ with $\mu(e)>0$. 
By concavity of $\log$, one has 
\[
H(\mu^{*n})=\sum_x \mu^{*n}(x) \log\frac{1}{\mu^{*n}(x)}
 \le \log \sum_{x\in\supp\mu^{*n}} \frac{\mu^{*n}(x)}{\mu^{*n}(x)} 
 =\log |\supp\mu^{*n}| . 
\]
Since $|\supp \mu^{*n}| = | (\supp\mu)^n |$ has weak polynomial growth, 
this implies that 
\[
d := \liminf_n \frac{H(\mu^{*n})}{\log n} <\infty.
\]
Now for any $d'<\liminf_n n(H(\mu^{*n+1})-H(\mu^{*n}))$, one has 
\[
H(\mu^{*n}) = \sum_{k=0}^{n-1}(H(\mu^{*k+1})-H(\mu^{*k})) 
\geq \mbox{const.} + \sum_{k=1}^{n-1}\frac{d'}{k} = \mbox{const.}  + d' \log n.
\]
Thus $d'\le d$, that is to say, $G$ has slow entropy growth. 
It is likely that the converse is also true. 
Indeed, suppose that $G$ does not have weakly polynomial growth. 
Then, for any $d$, one has $\liminf_n |S^n|/n^d=\infty$ and so, 
by Varopoulos's inequality (\cite{varopoulos}), 
$\mu^{*n}(e)=O(n^{-d/2})$. 
But since $\mu^{*n}(e)\geq\mu^{*n}(g)$ for every even $n$ and every $g\in G$ 
by the Cauchy--Schwarz inequality, this implies 
\[
H(\mu^{*n}) \geq \log\frac{1}{\mu^{*n}(e)} \geq \mbox{const.}+\frac{d}{2}\log n
\]
for even $n$. 
Since $d$ was arbitrary, one has $\lim_n H(\mu^{*n})/\log n = \infty$.

\section{Proof of Theorem}\label{sec:proof}

\begin{proof}[Proof of Theorem]
Let $b\colon G\to\cH$ be a $1$-cocycle with coefficients in an weakly mixing 
unitary representation $\pi$, and we will prove that $b\in\overline{B^1(G,\pi)}$.
As discussed in Section~\ref{sec:harmonic}, we may assume that $b$ 
is $\mu$-harmonic. Let $g\in\supp\mu$ and $\xi\in\cH$ be given. 
Then, since $b$ is $\mu^{*n}$-harmonic for every $n\in\IN$, one has for every $n$ 
\begin{align*}
\ip{b(g),\xi} = \ip{ \sum_x (b(gx)-b(x))\mu^{*n}(x), \xi}
 =\sum_x \ip{b(x),\xi} (g\mu^{*n}-\mu^{*n})(x), 
\end{align*}
and so, by inequalities (\ref{eq1}) and (\ref{eq2}),  
\begin{align*}
|\ip{b(g),\xi}|^2 &\le 8\delta(\mu^{*n},g\mu^{*n})\sum_x |\ip{b(x),\xi}|^2 (g\mu^{*n}+\mu^{*n})(x)\\
&\le \lambda_g (H(\mu^{*n+1}) - H(\mu^{*n})) \sum_x |\ip{b(x),\xi}|^2 (g\mu^{*n}+\mu^{*n})(x),
\end{align*}
where $\lambda_g=4\min\{\mu(e),\mu(g)\}^{-1}$. 
Since $\sum_x |\ip{b(x),\xi}|^2 (g\mu^{*n}+\mu^{*n})(x)$ has sublinear growth 
by Lemma in Section~\ref{sec:harmonic} 
(and the Cauchy--Schwarz inequality), 
slow entropy growth implies that 
$\ip{b(g),\xi}=0$. 
Since $g\in\supp\mu$ and $\xi\in\cH$ were arbitrary 
and $\mu$ is non-degenerate, this implies that $b=0$. 
\end{proof}

Note that the slow entropy growth condition implies the following (\cite[Lemma 8]{ek}) 
\[
\liminf_n n^{-1/2} \max_{|g|\le1} \| \mu^{*n} - g\mu^{*n} \|_1 <\infty.
\]
Since the slow entropy growth condition seems too restrictive, 
we give here a supplementary result that the above condition 
yields an weaker conclusion (although the author is still not aware of 
any super-polynomial growth group to which Proposition applies.) 

\begin{prop*}
Let $G$ be a finitely generated group with the word length $|\,\cdot\,|$ and 
let $\mu$ be a non-degenerate finitely-supported 
symmetric probability measure with $\mu(e)>0$.
Assume that there is $\delta>0$ such that for any $\epsilon>0$ and any
$N\in\IN$ there is $n\geq N$ such that for any $g\in G$ and any $E\subset G$ 
if $|g|\le \delta n^{1/2}$ and $\mu^{*n}(E)\geq 1-\delta$ then 
$\mu^{*n}(gEB_{\epsilon n^{1/2}})\geq\delta$. 
Here $B_r=\{ x : |x|\le r\}$. 
Then, any $1$-cocycle 
$b$ with coefficients in an weakly mixing unitary representation $\pi$ has sublinear growth 
in the sense that $\| b(g) \| \le f(|g|)$ for some $f$ with $f(n)/n\to0$.
In particular if $G$ moreover has a controlled F{\o}lner sequence, then $G$ has property $H_{\mathrm FD}$. 
\end{prop*}

\begin{proof}
Since approximate $1$-coboundaries have sublinear growth (\cite[Corollary 3.3]{ctv}), 
we may assume that $b$ is a $\mu$-harmonic cocycle such that $\| b(g) \|\le|g|$. 
Let $\gamma>0$ be given arbitrary and put $\epsilon := \gamma\delta^2$. 
By Lemma in Section~\ref{sec:harmonic}, there is $N$ such that $n\geq N$ implies 
\[
\sup_{\xi\in\cH,\, \|\xi\|\le1} \sum_x \mu^{*n}(x) |\ip{ b(x),\xi }|^2 < \epsilon^2 n.
\]
Take $n\geq \max\{N,4\delta^{-2}\}$ which fulfills the statement in Proposition.
Let $g\in G$ be such that $|g|\le \delta n^{1/2}$. 
For each unit vector $\xi\in \cH$, put 
\[
E_\xi:= \{ x \in G : |\ip{ b(x),\xi }| \le \epsilon n^{1/2}/\delta \}
\]
and observe that $\mu^{*n}(E_\xi)>1-\delta$. 
Hence, one has 
$\mu^{*n}(E_{\pi_g^*\xi}\cap g^{-1}E_\xi B_{\epsilon n^{1/2}})>0$ by assumption, 
and so there exist $x_\xi \in E_{\pi_g^*\xi}$ and $y_\xi\in B_{\epsilon n^{1/2}}$ 
such that $gx_\xi y_\xi \in E_\xi$. 
It follows that 
\begin{align*}
|\ip{ b(g),\xi }|
 &\le |\ip{ b(gx_\xi y_\xi),\xi }| + |\ip{ b(x_\xi),\pi_g^*\xi }| + |\ip{ b(y_\xi),\pi_{gx_\xi}^*\xi }|\\
 &\le 3\epsilon n^{1/2}/\delta = 3\gamma \delta n^{1/2}.
\end{align*}
This means that $\| b(g) \|\le 3\gamma \delta n^{1/2}$ 
for all $g\in G$ such that $|g|\le \delta n^{1/2}$, and so 
\[
\lim_m \frac{\max\{\|b(g)\| : |g| \le m\} }{m} 
 = \inf_m \frac{\max\{\|b(g)\| : |g| \le m\} }{m}
 \le \frac{3 \gamma \delta n^{1/2}}{\lfloor \delta n^{1/2} \rfloor}
 \le 6\gamma,
\]
where the first equality follows from subadditivity. 
Since $\gamma>0$ was arbitrary, this proves the first statement. 
The second follows from \cite[Corollary 3.7]{ctv}.
\end{proof}

\appendix
\section{Property (T) and harmonic $1$-cocycles}

We give a simple proof of the theorem of Mok and Korevaar--Schoen cited 
in the introduction (\cite{mok,ks}, see also \cite[Appendix A]{kleiner}). 
After submitting the first draft of this paper, the author learned that 
the same proof had been presented in Jesse Peterson's lecture 
at Vanderbilt University in Spring 2013. A more explicit construction 
(which still uses an ultrafilter) is provided later in \cite{eo}. 

\begin{thmA}\label{thmA}
Let $G$ be a finitely generated group without property (T). 
Then, for any non-degenerate finitely-supported symmetric probability 
measure $\mu$ on $G$, there is a nonzero $\mu$-harmonic $1$-cocycle with 
coefficients in some unitary representation.
\end{thmA}
\begin{proof}
Since $G$ does not have property (T), there is a unitary 
representation $\pi\colon G\acts\cH$ having approximate invariant vectors but having 
no nonzero invariant vectors (see \cite[Theorem 12.1.7]{bo}). 
Thus the self-adjoint contraction 
$T=\sum_x\mu(x)\pi_x \in \IB(\cH)$ contains $1$ in the spectrum, but 
not as an eigenvalue.
This means that $1$ is a limit point of the spectrum of $T$. 
Hence there is a sequence $\epsilon_n\searrow 0$ such that the spectral 
subspaces $\cH_n:=E_T([1-2\epsilon_n,1-\epsilon_n])\cH$ are nonzero. 
Take unit vectors $\xi_n\in\cH_n$. One has
\[
\sum_x\mu(x)\|\xi_n-\pi_x\xi_n\|^2=2(1-\ip{T\xi_n,\xi_n})\in [2\epsilon_n, 4\epsilon_n].
\]
Fix a free ultrafilter $\cU$ and consider the ultrapower unitary 
representation $\pi_{\cU}$ on the ultrapower Hilbert space $\cH_{\cU}$ 
(see \cite[12.1.4]{bo}).
Then the map $b\colon G\to\cH_{\cU}$, given by 
$b(x)=(\epsilon_n^{-1/2}(\xi_n - \pi_x\xi_n))_{n\to\cU}$, is a $1$-cocycle 
with coefficients in $\pi_{\cU}$ such that 
\[
\sum_x\mu(x)\| b(x) \|^2=\lim_{n\to\cU}\epsilon_n^{-1}\sum_x\mu(x)\|\xi_n - \pi_x\xi_n\|^2 \in [2,4]
\]
and 
\begin{align*}
\| \sum_x \mu(x)b(x) \|
 &= \lim_{n\to\cU} \epsilon_n^{-1/2}\| \sum_x \mu(x)(\xi_n - \pi_x\xi_n) \|\\
 &= \lim_{n\to\cU} \epsilon_n^{-1/2}\| \sum_x (1-T)\xi_n \| \le \lim_{n\to\cU} 2\epsilon_n^{1/2}=0.
\end{align*}
This means that $b$ is a nonzero $\mu$-harmonic $1$-cocycle. 
\end{proof}

\end{document}